\documentclass[a4paper,12pt,final]{amsart}
\usepackage{times,a4wide,mathrsfs,amssymb,amsmath,amsthm}

\newcommand{\C}{\mathbb{C}}

\newcommand{\LLL}{\mathbb{L}}
\newcommand{\QQ}{\mathbb{Q}}
\newcommand{\NN}{\mathbb{N}}
\newcommand{\PP}{\mathbb{P}}

\newcommand{\MM}{\mathcal M}

\newcommand{\gr}{\hbox{Gr}}

\newcommand{\rom}{\romannumeral}

\DeclareMathOperator{\aut}{Aut}

\newtheorem{theorem}{Theorem}[section]

\newtheorem{corollary}[theorem]{Corollary}
\newtheorem{proposition}[theorem]{Proposition}
\newtheorem{conjecture}[theorem]{Conjecture}
\newtheorem{remark}[theorem]{Remark}
\newtheorem{definition}[theorem]{Definition}
\newtheorem{convention}{Conventions}

\newtheorem{nonumbering}{Theorem}

\newtheorem{nonumberingc}{Corollary}

\newtheorem{nonumberingt}{Acknowledgements}

\begin{document}
\author[Robert Laterveer]
{Robert Laterveer}

\address{Institut de Recherche Math\'ematique Avanc\'ee,
CNRS -- Universit\'e 
de Strasbourg,\
7 Rue Ren\'e Des\-car\-tes, 67084 Strasbourg CEDEX,
FRANCE.}
\email{robert.laterveer@math.unistra.fr}

\title{On Voisin's conjecture for zero--cycles on hyperk\"ahler varieties}

\begin{abstract} Motivated by the Bloch--Beilinson conjectures, Voisin has made a conjecture concerning zero--cycles on self--products of Calabi--Yau varieties. We reformulate Voisin's conjecture in the setting of hyperk\"ahler varieties, and we prove this reformulated conjecture for one family of hyperk\"ahler fourfolds.
\end{abstract}


\maketitle

\section{Introduction}

Let $X$ be a smooth projective variety over $\C$, and let $A^i(X):=CH^i(X)_{\QQ}$ denote the Chow groups of $X$ (i.e. the groups of codimension $i$ algebraic cycles on $X$ with $\QQ$--coefficients, modulo rational equivalence). As is well--known, the field of algebraic cycles is rife with open questions \cite{B}, \cite{J2}, \cite{MNP}, \cite{Vo}. To wit, there is the following conjecture formulated by Voisin, which can be seen as a version of Bloch's conjecture for varieties of geometric genus one:

\begin{conjecture}[Voisin \cite{V9}]\label{conjvois} Let $X$ be a smooth projective complex variety of dimension $n$ with $h^{j,0}(X)=0$ for $0<j<n$ and $p_g(X)=1$.  
For any zero--cycles $a,a^\prime\in A^n(X)$ of degree zero, we have
  \[ a\times a^\prime=(-1)^n a^\prime\times a\ \ \ \hbox{in}\ A^{2n}(X\times X)\ .\]
  (Here $a\times a^\prime$ is short--hand for the cycle class $(p_1)^\ast(a)\cdot(p_2)^\ast(a^\prime)\in A^{2n}(X\times X)$, where $p_1, p_2$ denote projection on the first, resp. second factor.)
   \end{conjecture}
   
 Conjecture \ref{conjvois} is still wide open for a general $K3$ surface (on the positive side, cf. \cite{V9}, \cite{moi}, \cite{des}, \cite{tod} for some cases where this conjecture is verified).

Let us now suppose that $X$ is a hyperk\"ahler variety (i.e., a projective irreducible holomorphic symplectic manifold, cf. \cite{Beau0}, \cite{Beau1}). Conjecture \ref{conjvois} does not apply verbatim to $X$ (the Calabi--Yau condition 
$h^{j,0}(X)=0$ for $0<j<n$ is not satisfied), yet one can adapt conjecture \ref{conjvois} to make sense for $X$. For this adaptation, we will optimistically assume the Chow ring of $X$ has a bigraded ring structure $A^\ast_{(\ast)}(X)$, where each $A^i(X)$ splits into pieces
  \[ A^i(X) =\bigoplus_j A^i_{(j)}(X)\ .\]
 (Conjecturally, such a splitting exists for all hyperk\"ahler varieties, and the piece $A^i_{(j)}(X)$ should be isomorphic to the graded $\gr^j_F A^i(X)$ for the conjectural Bloch--Beilinson filtration \cite{Beau3}. Using the concept of ``(weak) multiplicative Chow--K\"unneth decompositions'', a bigraded ring structure has been constructed for Hilbert schemes of length $n$ subschemes of $K3$ surfaces \cite{V6}, for generalized Kummer varieties \cite{FTV}, and for Fano varieties of lines in very general cubic fourfolds \cite{SV}. For a different approach to the construction of a bigraded ring structure on the Chow ring of hyperk\"ahler varieties, cf. \cite{V14}.) 
 
 Since the piece $A^i_{(j)}(X)$ should be related to the cohomology group $H^{2i-j}(X)$, and 
   \[ \wedge^2  H^{s}(X)\ \subset\ H^{2s}(X\times X) \]
   should be supported on a divisor for any $s$ (in view of the generalized Hodge conjecture), we arrive at the following version of conjecture \ref{conjvois}:

\begin{conjecture}\label{conjHK} Let $X$ be a hyperk\"ahler variety of dimension $2m$. Let $a,a^\prime\in A^{2m}_{(j)}(X)$. Then
  \[ a\times a^\prime -a^\prime\times a=0\ \ \ \hbox{in}\ A^{4m}_{}(X\times X)\ .\]
  \end{conjecture}
  
  (We note that in conjecture \ref{conjHK}, we silently presuppose that $A^{2m}_{(j)}(X)=0$ for $j$ odd. This is expected to hold for any hyperk\"ahler variety, and is known unconditionally for the above--mentioned cases where a bigraded ring structure exists.)

  The main result of this note is that conjecture \ref{conjHK} is true for a certain family of hyperk\"ahler fourfolds:
  
  \begin{nonumbering}[=theorem \ref{main}] Let $Y\subset\PP^5(\C)$ be a smooth cubic fourfold defined by an equation
    \[ f(X_0,X_1,X_2)+ g(X_3,X_4,X_5)=0\ ,\]
    where $f$ and $g$ define smooth plane curves. Let $X=F(Y)$ be the Fano variety of lines in $Y$. Then for any $a,a^\prime\in A^4_{(j)}(X)$, there is equality
     \[ a\times a^\prime -a^\prime\times a=0\ \ \ \hbox{in}\ A^{8}_{}(X\times X)\ .\]
    \end{nonumbering} 
    
    Here the notation $A^\ast_{(\ast)}(X)$ refers to the Fourier decomposition of the Chow ring constructed by Shen--Vial \cite{SV}. (We mention in passing that for $X$ as in theorem \ref{main}, it is not yet known whether the pieces $A^\ast_{(\ast)}(X)$ fit together to form a bigraded ring, cf. remark \ref{pity}.) 
    
    The proof of theorem \ref{main} relies on the theory of finite--dimensional motives \cite{Kim}.
    
   As a corollary to theorem \ref{main}, we find that a certain instance of the generalized Hodge conjecture is verified on $X\times X$: 
   
   \begin{nonumberingc}[=corollary \ref{cor}] Let $X$ be as in theorem \ref{main}. The Hodge sub--structure
     \[ \wedge^2 H^4(X,\QQ)\ \subset\ H^8(X\times X,\QQ) \]
     is supported on a divisor.
   \end{nonumberingc}

 \vskip0.6cm

\begin{convention} In this article, the word {\sl variety\/} will refer to a reduced irreducible scheme of finite type over $\C$. A {\sl subvariety\/} is a (possibly reducible) reduced subscheme which is equidimensional. 

{\bf All Chow groups will be with rational coefficients}: we will denote by $A_j(X)$ the Chow group of $j$--dimensional cycles on $X$ with $\QQ$--coefficients; for $X$ smooth of dimension $n$ the notations $A_j(X)$ and $A^{n-j}(X)$ are used interchangeably. 

The notations $A^j_{hom}(X)$, $A^j_{AJ}(X)$ will be used to indicate the subgroups of homologically trivial, resp. Abel--Jacobi trivial cycles.
For a morphism $f\colon X\to Y$, we will write $\Gamma_f\in A_\ast(X\times Y)$ for the graph of $f$.
The contravariant category of Chow motives (i.e., pure motives with respect to rational equivalence as in \cite{Sc}, \cite{MNP}) will be denoted $\MM_{\rm rat}$.



We will write $H^j(X)$ 
to indicate singular cohomology $H^j(X,\QQ)$.

\end{convention}

\section{Preliminaries}

\subsection{Finite--dimensional motives}

We refer to \cite{Kim}, \cite{An}, \cite{Iv}, \cite{J4}, \cite{MNP} for the definition of finite--dimensional motive. 
An essential property of varieties with finite--dimensional motive is embodied by the nilpotence theorem:

\begin{theorem}[Kimura \cite{Kim}]\label{nilp} Let $X$ be a smooth projective variety of dimension $n$ with finite--dimensional motive. Let $\Gamma\in A^n(X\times X)_{}$ be a correspondence which is numerically trivial. Then there is $N\in\NN$ such that
     \[ \Gamma^{\circ N}=0\ \ \ \ \in A^n(X\times X)_{}\ .\]
\end{theorem}

 Actually, the nilpotence property (for all powers of $X$) could serve as an alternative definition of finite--dimensional motive, as shown by Jannsen \cite[Corollary 3.9]{J4}.
Conjecturally, any variety has finite--dimensional motive \cite{Kim}. We are still far from knowing this, but at least there are quite a few non--trivial examples.

\subsection{MCK decomposition}

\begin{definition}[Murre \cite{Mur}] Let $X$ be a projective quotient variety of dimension $n$. We say that $X$ has a {\em CK decomposition\/} if there exists a decomposition of the diagonal
   \[ \Delta_X= \pi_0+ \pi_1+\cdots +\pi_{2n}\ \ \ \hbox{in}\ A^n(X\times X)\ ,\]
  such that the $\pi_i$ are mutually orthogonal idempotents and $(\pi_i)_\ast H^\ast(X)= H^i(X)$.
  
  (NB: ``CK decomposition'' is shorthand for ``Chow--K\"unneth decomposition''.)
\end{definition}

\begin{remark}
\normalfont The existence of a CK decomposition for any smooth projective variety is part of Murre's conjectures \cite{Mur}, \cite{J2}. 
\end{remark}

\begin{definition}[Shen--Vial \cite{SV}]
\normalfont Let $X$ be a projective quotient variety of dimension $n$. Let $\Delta_X^{sm}\in A^{2n}(X\times X\times X)$ be the class of the small diagonal
  \[ \Delta_X^{sm}:=\bigl\{ (x,x,x)\ \vert\ x\in X\bigr\}\ \subset\ X\times X\times X\ .\]
  An {\em MCK decomposition\/} is a CK decomposition $\{\pi^X_i\}$ of $X$ that is {\em multiplicative\/}, i.e. it satisfies
  \[ \pi^X_k\circ \Delta_X^{sm}\circ (\pi^X_i\times \pi^X_j)=0\ \ \ \hbox{in}\ A^{2n}(X\times X\times X)\ \ \ \hbox{for\ all\ }i+j\not=k\ .\]
  
 (NB: ``MCK decomposition'' is shorthand for ``multiplicative Chow--K\"unneth decomposition''.) 
  
 A {\em weak MCK decomposition\/} is a CK decomposition $\{\pi^X_i\}$ of $X$ that satisfies
    \[ \Bigl(\pi^X_k\circ \Delta_X^{sm}\circ (\pi^X_i\times \pi^X_j)\Bigr){}_\ast (a\times b)=0 \ \ \ \hbox{for\ all\ } a,b\in\ A^\ast(X)\ .\]
  \end{definition}
  
  \begin{remark}
  \normalfont The small diagonal (seen as a correspondence from $X\times X$ to $X$) induces the {\em multiplication morphism\/}
    \[ \Delta_X^{sm}\colon\ \  h(X)\otimes h(X)\ \to\ h(X)\ \ \ \hbox{in}\ \MM_{\rm rat}\ .\]
 Suppose $X$ has a CK decomposition
  \[ h(X)=\bigoplus_{i=0}^{2n} h^i(X)\ \ \ \hbox{in}\ \MM_{\rm rat}\ .\]
  By definition, this decomposition is multiplicative if for any $i,j$ the composition
  \[ h^i(X)\otimes h^j(X)\ \to\ h(X)\otimes h(X)\ \xrightarrow{\Delta_X^{sm}}\ h(X)\ \ \ \hbox{in}\ \MM_{\rm rat}\]
  factors through $h^{i+j}(X)$.
  
  If $X$ has a weak MCK decomposition, then setting
    \[ A^i_{(j)}(X):= (\pi^X_{2i-j})_\ast A^i(X) \ ,\]
    one obtains a bigraded ring structure on the Chow ring: that is, the intersection product sends $A^i_{(j)}(X)\otimes A^{i^\prime}_{(j^\prime)}(X) $ to  $A^{i+i^\prime}_{(j+j^\prime)}(X)$.
    
      It is expected (but not proven !) that for any $X$ with a weak MCK decomposition, one has
    \[ A^i_{(j)}(X)\stackrel{??}{=}0\ \ \ \hbox{for}\ j<0\ ,\ \ \ A^i_{(0)}(X)\cap A^i_{hom}(X)\stackrel{??}{=}0\ ;\]
    this is related to Murre's conjectures B and D, that have been formulated for any CK decomposition \cite{Mur}.

  The property of having an MCK decomposition is severely restrictive, and is closely related to Beauville's ``(weak) splitting property'' \cite{Beau3}. For more ample discussion, and examples of varieties with an MCK decomposition, we refer to \cite[Section 8]{SV}, as well as \cite{V6}, \cite{SV2}, \cite{FTV}.
    \end{remark}

In what follows, we will make use of the following: 

\begin{theorem}[Shen--Vial \cite{SV}]\label{fanomck} Let $Y\subset\PP^5(\C)$ be a smooth cubic fourfold, and let $X:=F(Y)$ be the Fano variety of lines in $Y$. There exists a CK decomposition $\{\pi^X_i\}$ for $X$, and 
  \[ (\pi^X_{2i-j})_\ast A^i(X) = A^i_{(j)}(X)\ ,\]
  where the right--hand side denotes the splitting of the Chow groups defined in terms of the Fourier transform as in \cite[Theorem 2]{SV}. Moreover, we have
  \[ A^i_{(j)}(X)=0\ \ \ \hbox{for\ }j<0\ \hbox{and\ for\ }j>i\ .\]
  
  In case $Y$ is very general, the Fourier decomposition $A^\ast_{(\ast)}(X)$ forms a bigraded ring, and hence
  $\{\pi^X_i\}$ is a weak MCK decomposition.
    \end{theorem}

\begin{proof} (A remark on notation: what we denote $A^i_{(j)}(X)$ is denoted $CH^i(X)_j$ in \cite{SV}.)

The existence of a CK decomposition $\{\pi^X_i\}$ is \cite[Theorem 3.3]{SV}, combined with the results in \cite[Section 3]{SV} to ensure that the hypotheses of \cite[Theorem 3.3]{SV} are satisfied. According to \cite[Theorem 3.3]{SV}, the given CK decomposition agrees with the Fourier decomposition of the Chow groups. The ``moreover'' part is because the $\{\pi^X_i\}$ are shown to satisfy Murre's conjecture B \cite[Theorem 3.3]{SV}.

The statement for very general cubics is \cite[Theorem 3]{SV}.
    \end{proof}

\begin{remark}\label{pity}
\normalfont Unfortunately, it is not yet known that the Fourier decomposition of \cite{SV} induces a bigraded ring structure on the Chow ring for {\em all\/} Fano varieties of smooth cubic fourfolds. For one thing, it has not yet been proven that $A^2_{(0)}(X)\cdot A^2_{(0)}(X)\subset A^4_{(0)}(X)$ (cf. \cite[Section 22.3]{SV} for discussion).
\end{remark}



\subsection{Fano varieties of cubic fourfolds}

Let $X$ be the Fano variety of lines on a smooth cubic fourfold. As we have seen (theorem \ref{fanomck}), the Chow ring of $X$ splits into pieces $A^i_{(j)}(X)$.
The magnum opus \cite{SV} contains a detailed analysis of the multiplicative behaviour of these pieces. Here are the relevant results we will be needing:

\begin{theorem}[Shen--Vial \cite{SV}]\label{chowringfano} Let $Y\subset\PP^5(\C)$ be a smooth cubic fourfold, and let $X:=F(Y)$ be the Fano variety of lines in $Y$. 

\noindent
(\rom1) There exists $\ell\in A^2_{(0)}(X)$ such that intersecting with $\ell$ induces an isomorphism
  \[ \cdot\ell\colon\ \ \ A^2_{(2)}(X)\ \xrightarrow{\cong}\ A^4_{(2)}(X)\ .\]

\noindent
(\rom2) Intersection product induces a surjection
  \[ A^2_{(2)}(X)\otimes A^2_{(2)}(X)\ \twoheadrightarrow\ A^4_{(4)}(X)\ .\]
\end{theorem} 
     
 \begin{proof} Statement (\rom1) is \cite[Theorem 4]{SV}. Statement (\rom2) is \cite[Proposition 20.3]{SV}.
  \end{proof}

\section{Main result}

\begin{theorem}\label{main}  Let $Y\subset\PP^5(\C)$ be a smooth cubic fourfold defined by an equation
    \[ f(X_0,X_1,X_2)+ g(X_3,X_4,X_5)=0\ ,\]
    where $f$ and $g$ define smooth plane curves. Let $X=F(Y)$ be the Fano variety of lines in $Y$. Then for any $a,a^\prime\in A^4_{(j)}(X)$, there is equality
     \[ a\times a^\prime -a^\prime\times a=0\ \ \ \hbox{in}\ A^{8}_{(2j)}(X\times X)\ .\]   
  \end{theorem}     
 
 \begin{proof}    
  Let $u_Y\colon Y\to Y$ be the automorphism defined by the automorphism of $\PP^5(\C)$
    \[  [X_0\colon\cdots\colon X_5]\ \mapsto\ [X_0\colon X_1\colon X_2\colon \omega X_3\colon \omega X_4\colon \omega X_5]\ ,\]
    where $\omega$ is a primitive third root of unity. Let $u_X\in\aut(X)$ denote the automorphism induced by $u_Y$. As explained in \cite[Proof of Proposition 12]{Beau2}, $u_X$ is a symplectic automorphism, and the fixed locus $B\subset X$ of $u_X$ is isomorphic to $E\times E^\prime$, where $E$ and $E^\prime$ are the elliptic curves 
    \[ \begin{split} E&:= Y\cap \{ X_0=X_1=X_2=0\}\ ,\\
                       E^\prime&:= Y\cap \{ X_3=X_4=X_5=0\}\ .\\
                       \end{split}\]
  What's more, as $u_X$ is symplectic, $\tau\colon B\to X$ is the inclusion of a symplectic submanifold and so restriction induces an isomorphism
   \[    \tau^\ast    \colon\ \ H^{2,0}(X)\ \to\ H^{2,0}(B)\ .\]   
  This implies there is also an isomorphism of transcendental lattices
    \[    \tau^\ast    \colon\ \ H^{2}_{tr}(X)\ \to\ H^{2}_{tr}(B)\ \]
    (here $H^2_{tr}()\subset H^2()$ is defined as the smallest Hodge sub--structure containing $H^{2,0}$).   
   A bit more formally, this means there is an isomorphism of homological motives
   \begin{equation}\label{isomot} \Gamma= {}^t\Gamma_\tau+ \Gamma^\prime \colon\ \ \ (X,\pi^X_{2},0)\ \xrightarrow{\cong}\ (B,\pi^B_{2,tr},0) \oplus  \LLL^{\oplus \rho}\ \ \ \hbox{in}\ \MM_{\rm hom}\ .\end{equation}
  Here $\LLL$ denotes the Lefschetz motive, and $\rho$ is the dimension of the N\'eron--Severi group of $X$. The projector $\pi^X_2$ is as in theorem \ref{fanomck}, and $\pi^B_{2,tr}$ is a projector defining the transcendental part of the motive of a surface as in \cite{KMP}.
   
   As both sides of (\ref{isomot}) are finite--dimensional motives (for the left--hand side, this follows from proposition \ref{fanofindim}), we can upgrade the isomorphism (\ref{isomot}) to an isomorphism of Chow motives
   \begin{equation}\label{isomotch} \Gamma\colon\ \ \ (X,\pi^X_{2},0)\ \xrightarrow{\cong}\ (B,\pi^B_{2,tr},0)\oplus \LLL^{\oplus \rho}\ \ \ \hbox{in}\ \MM_{\rm rat}\ .\end{equation}  
               Since isomorphic Chow motives have isomorphic Chow groups, the isomorphism (\ref{isomotch}) implies there is an induced isomorphism
      \[ (\pi^B_{2,tr} \circ \Gamma \circ \pi^X_2)_\ast  \colon\ \ \ (\pi^X_2)_\ast A^2_{hom}(X)\ \xrightarrow{\cong}\  (\pi^B_{2,tr})_\ast A^2_{hom}(B)\oplus 
      A^2_{hom}(\LLL^{\oplus \rho})\ .\]
      But $A^2_{(2)}()\subset A^2_{hom}()$, and $A^2_{hom}(\LLL)=0$, and so we find that   
            restriction induces an isomorphism
   \[  (\pi^B_{2,tr}\circ {}^t \Gamma_\tau)_\ast\colon\ \ \ A^2_{(2)}(X)\ \xrightarrow{\cong}\  ( \pi^B_{2,tr})_\ast A^2_{hom}(B)=A^2_{(2)}(B)=A^2_{AJ}(B)\ .\]
   (Here, the notation $A^2_{(2)}(B)$ refers to Beauville's splitting of the Chow ring of abelian varieties \cite{Beau}.)
   
  We are now in position to state a result that will serve as an intermediate step towards proving theorem \ref{main}:
  
  \begin{proposition}\label{intermed} Let $X$ be as in theorem \ref{main}. For any $b,b^\prime\in A^2_{(2)}(X)$, there is equality
   \[  b\times b^\prime - b^\prime\times b=0\ \ \ \hbox{in}\ A^4(X\times X)\ .\]
   \end{proposition}
   
   \begin{proof} Let $\Psi^\prime\in A^{2}(B\times X)$ be inverse to $ \Gamma$ in the isomorphism (\ref{isomotch}), and let
    \[ \Psi:=   \pi^X_{2}\circ \Psi^\prime\circ \pi^B_{2,tr}\ \ \ \in A^{2}(B\times X)\ .\]
    It follows formally that $\Psi$ induces isomorphisms
    \[ \Psi_\ast\colon\ \ \ A^2_{(2)}(B)\ \xrightarrow{\cong}\ A^2_{(2)}(X)\ .\]
    There is a commutative diagram
    \[ \begin{array}
       [c]{ccc}
       A^2_{(2)}(B)\otimes A^2_{(2)}(B) & \xrightarrow{\times} &A^4_{(4)}(B\times B)\\
      \ \ \ \ \ \ \ \ \ \ \ \ \ \  \downarrow{\scriptstyle (\Psi_\ast,\Psi_\ast)} && \ \ \ \ \ \ \ \  \ \ \ \ \downarrow{\scriptstyle (\Psi\times \Psi)_\ast}\\
       A^2_{(2)}(X)\otimes A^2_{(2)}(X) & \xrightarrow{\times} &A^4_{}(X\times X)\\
       \end{array}\]
       and (as we have just seen) the left--hand vertical arrow is an isomorphism. To prove proposition \ref{intermed}, it thus suffices to prove the corresponding statement for the abelian surface $B$; this is taken care of by a result of Voisin's:
       
     \begin{theorem}[Voisin \cite{Vo}]\label{thabvar} Let $B$ be an abelian variety of dimension $g$. For any $c,c^\prime\in A^g_{(g)}(B)$, there is equality
     \[  c\times c^\prime -(-1)^g\ c^\prime\times c=0\ \ \ \hbox{in}\ A^{2g}_{(2g)}(B\times B)\ .\]
     \end{theorem}
       
    (This is \cite[Example 4.40]{Vo}. The notation $A^\ast_{(\ast)}$ refers to Beauville's splitting of the Chow ring of an abelian variety \cite{Beau}.)
    
    To go from proposition \ref{intermed} to theorem \ref{main}, we apply theorem \ref{chowringfano}. Let us first suppose $j=2$, i.e. $a,a^\prime$ are zero--cycles in $A^4_{(2)}(X)$. According to theorem \ref{chowringfano}(\rom1), we can express $a$ and $a^\prime$ as
     \[   a= b\cdot \ell\ \ \ , a^\prime=b^\prime\cdot\ell\ ,\]
     for some (unique) $b,b^\prime\in A^2_{(2)}(X)$. It follows that
     \[ \begin{split} a\times a^\prime - a^\prime\times a &= (b\cdot\ell)\times (b^\prime\cdot\ell) - (b^\prime\cdot\ell)\times (b\cdot\ell)\\
                                                                                 &= (b\times b^\prime - b^\prime\times b)\cdot (\ell\times\ell)\\
                                                                                 &=0\ \ \  \hbox{in}  \ A^8(X\times X)\ ,\\
                                                                          \end{split}\]
                                                              where the last line follows from proposition \ref{intermed}. 
                                                              
          Next, let us suppose $j=4$, i.e. $a,a^\prime$ are zero--cycles in 
                                                              $A^4_{(4)}(X)$. According to theorem \ref{chowringfano}(\rom2), we can express $a$ and $a^\prime$ as
   \[ a=b_1\cdot b_2\ \ \ ,\ a^\prime=b_1^\prime\cdot b_2^\prime\ ,\]
   for some $b_1,b_2,b_1^\prime,b_2^\prime\in A^2_{(2)}(X)$. It follows that
   \[ \begin{split} a\times a^\prime  &= (b_1\cdot b_2)\times (b_1^\prime\cdot b_2^\prime)     \\
                                                                                 &= (b_1\times b_1^\prime)\cdot   (b_2\times b_2^\prime) \\
                                                                                 &= (b_1^\prime\times b_1)\cdot   (b_2^\prime\times b_2) \\
                                                                                 &= (b_1^\prime\cdot b_2^\prime) \times (b_1\cdot b_2)\\                                                                                 
                                                                                 &=   a^\prime\times a   \ \ \  \hbox{in}  \ A^8(X\times X)\ ,\\
                                                                       \end{split}\]
                                                                       where in the middle we have used proposition \ref{intermed}.
                                         This settles the case $j=4$. The remaining case $j=0$ is trivially true since $A^4_{(0)}(X)\cong\QQ$.          
    \end{proof}

 \begin{proposition}\label{fanofindim} Let $Y\subset\PP^5(\C)$ be a cubic as in theorem \ref{main}, and let $X=F(Y)$ be the Fano variety of lines in $Y$. Then $Y$ and $X$ have finite--dimensional motive.
  \end{proposition}  
    
 \begin{proof} To establish finite--dimensionality of $Y$ is an easy exercice in using what is commonly known as the ``Shioda inductive structure'' \cite{Shi}, \cite{KS}. Indeed, applying 
 \cite[Remark 1.10]{KS}, we find there exists a dominant rational map
   \[ \phi\colon\ \ \ Y_1\times Y_2\ \dashrightarrow\ Y\ ,\]
   where $Y_1, Y_2\subset\PP^3(\C)$ are smooth cubic surfaces defined as
     \[ \begin{split} &f(X_0,X_1,X_2)+V^3=0\ ,\\
                          &g(X_0,X_1,X_2)+W^3=0\ .\\
                       \end{split}\]
           The indeterminacy locus of $\phi$ is resolved by blowing up the locus $C_1\times C_2\subset Y_1\times Y_2$, where $C_1, C_2$ are cubic curves. As $Y_1,Y_2,C_1,C_2$ have finite--dimensional motive, the blow-up also has finite--dimensional motive. Since this blow--up dominates $Y$, the cubic $Y$ has finite--dimensional motive. 
           
       The main result of \cite{fano} now implies that the Fano variety $X=F(Y)$ also has finite--dimensional motive.            
        \end{proof}   
    
  \end{proof}

 \begin{corollary}\label{cor} Let $X$ be as in theorem \ref{main}. The Hodge sub--structure
     \[ \wedge^2 H^4(X)\ \subset\ H^8(X\times X) \]
     is supported on a divisor.
   \end{corollary}

 \begin{proof} As noted in \cite{V9}, this follows Bloch--Srinivas style from the truth of conjecture \ref{conjvois}. The idea is as follows. Let 
   \[\iota\colon X\times X\to X\times X\]
   denote the involution switching the two factors. The correspondence
   \[ \wedge^2 \pi_4^X := \bigl( \Delta_{X\times X} -\Gamma_\iota\bigr)\circ (\pi^X_4\times \pi^X_4)\ \ \ \in A^8(X^4) \]
   acts on cohomology as a projector on $\wedge^2 H^4(X)\ \subset\ H^8(X\times X)$. 
   On the other hand, the $j=4$ case of theorem \ref{main} implies that
     \[  (\wedge^2 \pi^X_4)_\ast A^8(X\times X) =0\ .\]
   Using the Bloch--Srinivas argument \cite{BS}, this implies the correspondence $\wedge^2 \pi^X_4$ is rationally equivalent to a cycle supported on $X\times X\times D$ where $D\subset X\times X$ is a divisor. Returning to the action of $\wedge^2 \pi^X_4$ on cohomology, this implies
   \[    \wedge^2 H^4(X)=(   \wedge^2 \pi^X_4)_\ast H^8(X\times X) \]
   is supported on $D$.    
 \end{proof}
 
 \begin{remark}
 \normalfont Using the work of Rie\ss\, \cite{Rie}, it is possible to extend theorem \ref{main} to all hyperk\"ahler varieties birational to the Fano variety $X$ of theorem \ref{main}.
 \end{remark}
 
 \begin{remark}\normalfont The family of Fano varieties of theorem \ref{main} has been much studied, starting with Namikawa's work \cite{Nam}. It also appears in \cite{Kawa}, and in
  \cite[Proposition 12]{Beau2} (where it is proven that for appropriate choices of $f$ and $g$, the Fano variety $X$ is ``$N^2$--maximal'' in the sense that
   \[ \dim_{\QQ} N^2 H^4(X) =\dim_\C H^{2,2}(X,\C)\ .)\] 
  \end{remark} 
 
 \begin{remark}\normalfont As the expert reader will have noticed, the crux of the proof of theorem \ref{main} is the existence of a symplectic automorphism of $X$ for which the fixed locus is an abelian surface. As such, it is natural to ask whether there are other families of Fano varieties of cubic fourfolds with this property. 
 
 However, if one adds the condition that the automorphism is of primary order and polarized, the Fano varieties as in theorem \ref{main} are the only Fano varieties of cubic fourfolds with this property (this follows from the classification given in \cite{LFu3}; the case studied in theorem \ref{main} appears as ``family IV-2'' in \cite[Theorem 0.1]{LFu3}, and there are no other instances where the fixed locus is an abelian surface).
 
 \end{remark}

\vskip1cm
\begin{nonumberingt} Thanks to all participants of the Strasbourg 2014/2015 ``groupe de travail'' based on the monograph \cite{Vo} for a very stimulating atmosphere.
Many thanks to Kai and Len for always respecting the rule not to bring toys into my office.
\end{nonumberingt}

\end{document}